\documentclass[10pt]{amsart}

\usepackage{mathrsfs}
\usepackage{enumerate}
\newtheorem{theorem}{Theorem}[section]

\newtheorem{lemma}[theorem]{Lemma}
\newtheorem{proposition}[theorem]{Proposition}
\newtheorem{corollary}[theorem]{Corollary}

\theoremstyle{definition}
\newtheorem{definition}[theorem]{Definition}
\newtheorem{example}[theorem]{Example}
\newtheorem{remark}[theorem]{Remark}

\def\LL{\mathscr{L}}

\def\NN{\mathbb{N}}

\def\ZZ{\mathbb{Z}}

\def\ee{\mathrm{e}}

\def\RR{\mathbb{R}}
\def\CC{\mathbb{C}}
\def\Cb{\mathrm{C}_{\mathrm{b}}}
\def\Ell{\mathrm{L}}
\def\co{\mathrm{c}_0}
\def\BUC{\mathrm{BUC}}
\def\ran{\mathop{\mathrm{ran}}}
\def\clconv{\mathop{\overline{\mathrm{conv}}}}

\def\ap{\mathrm{ap}}
\def\aap{\mathrm{aap}}

\def\LLL{\mathscr{L}}

\def\veps{\varepsilon}

\def\calS{\mathscr{S}}
\def\Id{\mathrm{I}}
\begin{document}
\title{A note on the periodic decomposition problem for semigroups}
\author{B\'alint Farkas}

\address{
  University of Wuppertal, Faculty of Mathematics and Natural Sciences, Gau\ss
  strasse 20, 42119, Wuppertal, Germany}
\email{farkas@math.uni-wuppertal.de}
\begin{abstract}

Given $T_1,\dots, T_n$ commuting power-bounded operators on a Banach space we study under which conditions the equality $\ker (T_1-\Id)\cdots (T_n-\Id)=\ker(T_1-\Id)+\cdots +\ker (T_n-\Id)$ holds true. This problem, known as the periodic decomposition problem, goes back to I.{} Z.{} Ruzsa. In this short note we consider the  case when $T_j=T(t_j)$, $t_j>0$, $j=1,\dots, n$ for some one-parameter semigroup $(T(t))_{t\geq 0}$. We also look at a generalization of the periodic decomposition problem when instead of the cyclic semigroups $\{T_j^n:n \in \NN\}$ more general semigroups of bounded linear operators are considered.
\end{abstract}

\thanks{The author was supported by the Hungarian Research Fund (OTKA-100461).}
\keywords{periodic decomposition problem, $C_0$-(semi)group, norm-continuous one-parameter semigroup, amenable semigroup}
\subjclass[2010]{47D06, 47B39, 47D99, 46B20}

\maketitle

The original decomposition problem as formulated by I.{} Z.{} Ruzsa concerns decompositions of functions $f:\RR\to \RR$  into finite sum of  periodic ones with prescribed periods.  
Let $a_1,a_2,\dots, a_n\in \RR$ be fixed real numbers. Given a function $f:\RR\to\RR$ decide whether it is possible to decompose it as a sum of periodic functions 
\begin{align}\label{eq:perdec}
f=f_1+\cdots+f_n \quad \text {with $f_j$ being $a_j$-periodic for $j=1,\dots,n$}.
\end{align}
 Denoting by $S_a:\RR^\RR\to \RR^\RR$ the shift operator $S_af(x)=f(x+a)$ a necessary condition for  the existence of such periodic decompositions  is easily formulated:
\begin{align}\label{eq:diff}
(S_{a_1}-\Id)\cdots(S_{a_n}-\Id)f=0
\end{align}
must hold, where $\Id$ denotes the identity operator, since the shift operators commute and annihilate the respective periodic functions. 
The \emph{periodic decomposition problem} asks for the converse direction: Must a function $f\in\RR^\RR$ subject to \eqref{eq:diff} be necessarily the sum of $a_j$-periodic functions as in \eqref{eq:perdec}? Or if not, then what difference equations (if any) characterize the existence of such periodic decompositions? This question was answered satisfactorily in \cite{FHKR} for the class of functions $\RR^\RR$, where a complete characterization of the existence  periodic decompositions is given in terms of difference equations of type \eqref{eq:diff}. It turned out that for  $a_1,\dots, a_n$  pairwise incommensurable condition \eqref{eq:diff} is  also sufficient for the existence of periodic decompositions, see also \cite{invari}. Moreover, in \cite{FHKR} not only translations on $\RR$ but also actions of mappings on an arbitrary set $X$ were considered.

\medskip\noindent
Of course, the same question is interesting  for translation invariant linear subspaces $E$ of $\RR^\RR$. In this setting one is interested in whether the equality
\begin{equation}\label{eq:translkern}
 \ker(S_{a_1}-\Id)\cdots (S_{a_n}-\Id)=\ker(S_{a_1}-\Id)+\cdots+\ker (S_{a_n}-\Id)
\end{equation}
holds true.  Here the inclusion ``$\supseteq$'' corresponds to aforementioned implication ``(1)$\Rightarrow$(2)''. The first result in this direction is due to Laczkovich and R\'ev\'esz, \cite{laczkovich/revesz:1989}, who showed that for $E=\Cb(\RR)$, the space of bounded and continuous functions, in \eqref{eq:translkern} indeed equality holds. The above reformulation opened up the toolbox of operator theory, and it was soon discovered that in many translation invariant subspaces of $\RR^\RR$, most notably in $E=\Ell^p(\RR)$, $p\in[1,\infty]$, the periodic decomposition problem can be answered affirmatively, see R\'ev\'esz, Laczkovich \cite{laczkovich/revesz:1990}. Leaving behind the original setting they considered the following purely operator theoretic problem: Given $E$ a Banach space and  $T_1,\dots, T_n\in \LL(E)$ bounded linear operators under which circumstances does the equality
\begin{equation}\label{eq:kereq}
\ker(T_{1}-\Id)\cdots (T_{n}-\Id)=\ker(T_{1}-\Id)+\cdots+\ker (T_{n}-\Id)
\end{equation}
hold true? Among others they have proved that this is so for power-bounded operators on reflexive spaces. Note that in this formulation it becomes apparent that whether the occurring functions are real or complex valued should be immaterial. This is not immediately clear if one looks at the original proof of Laczkovich and R\'ev\'esz for the decomposition property of $\Cb(\RR)$ with respect to translations. Gajda \cite{gajda:1992} proved the analogous result for the space $\BUC(\RR)$ of bounded and continuous functions. Some years later Kadets and Shumyatskiy \cite{KSadd, KSave} took up the topic again and provided some additional insights how various functional analytic/operator theoretic properties are related to the subject matter. 

\medskip\noindent  Instead of  the actions of $\NN$ or $\ZZ$, e.g., as translations $x\to x+na_j$  on $\RR$ one can more generally consider  the action of some semigroups $S_j$, $j=1\dots,n$ on some set $X$, and ask for the existence of decompositions of a function $f:X\to \CC$ as a sum
\begin{equation*}
f=f_1+\cdots+f_n
\end{equation*}
where $f_j$ are invariant under the action of $S_j$. This note intends to provide some first results in this general direction, see Section \ref{sec:sgrp} where we extend some results of Gajda \cite{Gajda}. While in Section \ref{sec:c0} we study the validity of \eqref{eq:kereq} when $T_j=T(t_j)$ for some one-parameter $C_0$-(semi)group $(T(t))_{t\geq 0}$, a first such result is due to Kadets and Shumyatskiy for $n=2$ and $(T(t))_{t\in \RR}$ a bounded group, see \cite{KSadd}.

\section{Commuting actions of semigroups}\label{sec:sgrp}
Let $S$ be an arbitrary (discrete) semigroup with unit element $1$, and let $S_j$, $j=1,\dots,n$ be commuting unital subsemigroups of $S$ that all act on the nonempty set $X$, the unit acting as the identity. For definiteness we write the action to the right, i.e., $st$ acts as $x\mapsto (xs)t$. For $s\in S$ the \emph{Koopman operator} $S_s:\CC^X\to\CC^X$ is defined by
\begin{equation*}
S_sf(x)=f(xs).
\end{equation*}

Generalizing \eqref{eq:kereq} for a bounded function $f:X\to\CC$ we ask whether the validity of the difference equation
\begin{align}\label{eq:Sdiff}
(S_{s_1}-\Id)\cdots (S_{s_n}-\Id)f=0
\quad\mbox{for all $s_j\in S_j$, $j=1,\dots,n$}
\end{align}
is equivalent to the existence of a decomposition
\begin{align}\label{eq:Sdecomp}
f=f_1+\cdots +f_n,
\end{align}
where for $j=1,\dots,n$ the function $f_j:X\to \CC$ is $S_j$-invariant, meaning that $S_{s_j}f=f$ for all  $s_j\in S_j$. Such a decomposition we shall call $({S}_1,{S}_2,\dots, {S}_n)$-invariant, whose existence implies \eqref{eq:Sdiff} without any further assumptions:  The operators $S_{s_j}-\Id$, $s_j\in S_j$ , $j=1,\dots,n$ commute with each other and an ${S}_j$-invariant function $f_j$ appearing in \eqref{eq:Sdecomp} is annihilated by $S_{s_j}-\Id$. If the semigroups $S_j$ are (left-)amenable we can prove the converse implication.

\begin{theorem}\label{thm:Sdecomp}
Suppose that for $j=1,\dots,n$ the unital semigroups $S_j$ acting on the set $X$ are left-amenable and the actions of $S_j$s are commuting. Then for every bounded function $f:X\to\CC$ the validity of the difference equation \eqref{eq:Sdiff} and the existence of an $(S_1,S_2,\dots,S_n)$-invariant decomposition \eqref{eq:Sdecomp} are equivalent.
\end{theorem}
\begin{proof}
We only need to prove  that \eqref{eq:Sdiff} implies \eqref{eq:Sdecomp}.
For $j=1,\dots,n$ let $\mu_j$ be a left-invariant mean on $\ell^\infty(S_j)$.
Let $x\in X$ be fixed. The difference equation
\eqref{eq:Sdiff} implies
\begin{equation}\label{eq:Sumeq}
0=\sum_{\epsilon\in 2^n} (-1)^{|\epsilon|}
f(xs_1^{\epsilon_1}s_2^{\epsilon_2}\cdots s_n^{\epsilon_n})
\end{equation}
for $s_i\in S_i$ and $i=1,\dots,n$. Note that by assumption (the actions of the $S_j$s  commute) the
order of the products is here immaterial. 
Let 
\begin{equation*}
h_{x,\epsilon}(s_1,s_2,\dots, s_n):=f(xs_1^{\epsilon_1}s_2^{\epsilon_2}\cdots s_n^{\epsilon_n}),
\end{equation*}
then, by fixing variables, each of the partial functions $s_j\mapsto h_{x,\epsilon}(s_1,s_2,\dots, s_n)$ belongs to $\ell^\infty(S_j)$ and so do the functions $s_j\mapsto \mu_i(s_i\mapsto h_{x,\epsilon}(s_1,s_2,\dots, s_n))$ for all $i,j=1,\dots,n$. Therefore, by applying the $\mu_j$s consecutively to \eqref{eq:Sumeq} we obtain the equality
\begin{equation}\label{eq:0means}
0=\sum_{\epsilon\in 2^n} (-1)^{|\epsilon|}
\mu_1(\mu_2(\cdots \mu_n(f(xs_1^{\epsilon_1}s_2^{\epsilon_2}\cdots
s_n^{\epsilon_n})\cdots))),
\end{equation}
where the action of $\mu_j$ is understood in the variable $s_j\in S_j$. Let us define the functions
\begin{align*}
f_j(x)&:=-\sum_{{\epsilon\in 2^n\atop \epsilon_j\neq 0}\atop
\epsilon_k=0\text{ for }k>j} (-1)^{|\epsilon|} \mu_1(\mu_2(\cdots
\mu_n(f(xs_1^{\epsilon_1}s_2^{\epsilon_2}\cdots
s_n^{\epsilon_n})\cdots)))\\
&=-\sum_{{\epsilon\in 2^n\atop \epsilon_j\neq 0}\atop
\epsilon_k=0\text{ for }k>j} (-1)^{|\epsilon|} \mu_1(\mu_2(\cdots
\mu_j(f(xs_1^{\epsilon_1}s_2^{\epsilon_2}\cdots s_j)\cdots))).
\end{align*}
In this way, from \eqref{eq:0means} we obtain
\begin{equation*}
f(x)=f_1(x)+f_2(x)+\cdots+f_n(x).
\end{equation*}
We prove that the function $f_j$ is $S_j$-invariant. For $s\in S_j$ and $i\neq j$ we can interchange $s_i$ with $s$ and obtain 
\begin{align*}
f_j(xs)&=\sum_{{\epsilon\in 2^n\atop \epsilon_j\neq 0}\atop
\epsilon_k=0\text{ for }k>j} (-1)^{|\epsilon|} \mu_1(\mu_2(\cdots
\mu_j(f(xss_1^{\epsilon_1}s_2^{\epsilon_2}\cdots
s_j)\cdots)))\\
&=\sum_{{\epsilon\in 2^n\atop \epsilon_j\neq 0}\atop
\epsilon_k=0\text{ for }k>j} (-1)^{|\epsilon|} \mu_1(\mu_2(\cdots
\mu_j(f(xs_1^{\epsilon_1}s_2^{\epsilon_2}\cdots
ss_j)\cdots)))\\
 &=\sum_{{\epsilon\in 2^n\atop \epsilon_j\neq 0}\atop
\epsilon_k=0\text{ for }k>j} (-1)^{|\epsilon|} \mu_1(\mu_2(\cdots
\mu_j(f(xs_1^{\epsilon_1}s_2^{\epsilon_2}\cdots
s_j)\cdots)))=f_j(x),
\end{align*}
where we used that $\mu_j$ is left-invariant.
\end{proof}

The result above and its proof is a generalization of one of Gajda. In \cite{Gajda} he applied Banach limits to prove the equivalence of \eqref{eq:Sdiff} and \eqref{eq:Sdecomp} in the original setting of the periodic decomposition problem, i.e., when $S_j=\ZZ$, $j=1,\dots,n$.

\begin{theorem}
Let $(X,d)$ be a metric space and suppose that the action of $S_j$ on $X$ is uniformly equicontinuous, i.e., for each $s_j\in S_j$, $j=1,\dots,n$ the mapping $x\mapsto xs_j$ is uniformly equicontinuous in $s_j\in S_j$. Let $f:X\to\CC$ be a bounded, uniformly continuous function. Then the difference equation \eqref{eq:Sdiff} and the existence of a decomposition as in \eqref{eq:Sdecomp} with $f_j$ uniformly continuous are equivalent.
\end{theorem}
\begin{proof}
We only need to prove that under the validity of \eqref{eq:Sdiff} the decomposition given in  the proof of Theorem \ref{thm:Sdecomp} yields uniformly continuous functions $f_j$. Let $\veps>0$ be given, and set $\delta_0:=\veps/2^n$. Let $\delta_i>0$, $i=1,2,\dots, j$ be chosen such that for $d(x,y)<\delta_{i}$ implies $d(xs_i,ys_i)<\delta_{i-1}$ for all $s_i\in S_i$ and $i=1,2,\dots,j$ (use uniform equicontinuity).
Set $\delta:=\delta_j$. Then for $d(x,y)<\delta$ one has
\begin{align*}
&|f_j(x)-f_j(y)|\\
&\leq\quad\sum_{{\epsilon\in 2^n\atop \epsilon_j\neq 0}\atop
\epsilon_k=0\text{ for }k>j} \Bigl|\mu_1\bigl(\mu_2\bigl(\cdots
\mu_n(f(xs_1^{\epsilon_1}s_2^{\epsilon_2}\cdots
s_n^{\epsilon_n})-f(ys_1^{\epsilon_1}s_2^{\epsilon_2}\cdots
s_n^{\epsilon_n}))\cdots\bigr)\bigr)\Bigr|\\
&\quad\leq2^n \delta_0\leq\veps,
\end{align*}
whenever $d(x,y)<\delta$.
\end{proof}
The same argumentation applies when $X$ is a uniform space.

\medskip\noindent We now turn to a pure operator theoretic formulation concerning  subsemigroups $\mathscr{S}\subseteq \LL(E)$ of the semigroup $\LL(E)$ of bounded linear operators on a Banach space $E$. Recall that a semigroup $\mathscr{S}\subseteq \LL(E)$ is called \emph{mean-ergodic} if the closed convex hull $\clconv{\calS}\subseteq \LL(E)$ contains a zero element $P$, i.e., $PT=P=TP$ for every $T\in \mathscr{S}$. In this case $P$ is a projection, called the mean ergodic projection of $\mathscr{S}$, and one has 
\begin{equation*}\ran P=\bigcap_{S\in \mathscr{S}} \ker(S-\Id)
\end{equation*} and $E=\ran P\oplus \ran(\Id-P)$,  see \cite{Nagel73}. For simplicity we introduce the notation $\ker\mathscr{A}=\bigcap_{S\in \mathscr{A}} \ker S$ for subsets $\mathscr{A}\subseteq \LL(E)$.
\begin{theorem}
Let $\mathscr{S}_1,\mathscr{S}_2,\dots,\mathscr{S}_n\subset \LL(E)$
be mean ergodic operator {semigroups} and suppose that $ST=TS$
whenever $T\in \mathscr{S}_i$, $S\in \mathscr{S}_j$ with $i\neq j$. Then
\begin{equation*} \ker (\mathscr{S}_1-\Id)\cdots (\mathscr{S}_n-\Id)=\ker
(\mathscr{S}_1-\Id)+\cdots +\ker(\mathscr{S}_n-\Id).\end{equation*}
\end{theorem}
\begin{proof}
Let $P_i$ be the mean ergodic projection onto $\ker (\mathscr{S}_i-\Id)$, recall
$P_i\in \clconv (\mathscr{S}_i)$, where the closure is understood in
the strong operator topology. As for any projection, we have
\begin{equation*}
E=\ran P_i\oplus \ker P_i=\ran P_i\oplus \ran (P_i-\Id).
\end{equation*}
Here $\ran P_i=\ker (\mathscr{S}_i-\Id)$. Since $\mathscr{S}_i$ and
$\mathscr{S}_j$ commute the subspaces in the above decomposition are
$\mathscr{S}_j$ invariant. So the proof
will be complete if we can show that for any
$\mathscr{S}_i$-invariant subspace $Y$ also the restricted semigroup
$\mathscr{S}_i|_Y$ is mean ergodic and if $x\in \ker
(\mathscr{S}_1-\Id)\cdots (\mathscr{S}_n-\Id)$ implies
$x\in \ker (P_1-\Id)\cdots (I-P_n)$. The latter assertion is
indeed true: If $x\in \ker (\mathscr{S}_1-\Id)\cdots
(\mathscr{S}_n-\Id)$, then  $x\in \ker
(\clconv(\mathscr{S}_1)-\Id)\cdots
(\clconv(\mathscr{S}_n)-\Id)$ by continuity. From this $x\in \ker
(P_1-\Id)\cdots (P_n-\Id)$ follows. Mean ergodicity of
$\mathscr{S}_j|_Y$ follows, since $Q_j=P_j|_Y\in
\clconv{\mathscr{S}_j|_Y}$ by the closedness of $Y$, and then $Q_j$ is the
mean ergodic projection of $\mathscr{S}_j|_Y$.
\end{proof}
One obtains the following corollary that was already present in \cite{laczkovich/revesz:1990}, but remained unnoticed behind other results. Note also that in \cite{KSave}  Kadets and Shumyatskiy prove also the following in  a slightly more general form.

\begin{corollary}\label{cor:mearg}
Let $E$ be a Banach space and $T_1,T_2, \dots, T_n$ commuting
mean ergodic operators (i.e., $\{T^n:n\in \NN\}$ is a mean-ergodic semigroup).
Then the equality
\begin{equation*}
\ker (T_1-\Id)\cdots (T_n-\Id)=\ker
(T_1-\Id)+\cdots+\ker (T_n-\Id)
\end{equation*}
holds.
\end{corollary}
\begin{example}\label{examp:cpt}
If for all $x\in E$ the orbits $\{T_j^nx:n\in \NN\}$ are weakly (or strongly) relatively compact then the above Corollary \ref{cor:mearg} applies, see, e.g., \cite[Sec.{} 8.4]{EFHN}.
\end{example}
\section{One-parameter semigroups}\label{sec:c0}
We turn our attention to operator semigroups that are continuously parametrized by $\RR_+$.
A \emph{$C_0$-semigroup} on a Banach space $E$ is a unital semigroup homomorphism $T:\RR_+\to \LLL(E)$ that is continuous with respect to the strong operator topology on $\LLL(E)$, denoted often as $(T(t))_{t\geq0}$. A \emph{$C_0$-group} is a $C_0$-semigroup consisting of continuously invertible operators, hence extending to a strongly continuous group homomorphism $T:\RR\to \LLL(E)$, and giving rise to the backward semigroup $(T^{-1}(t))_{t\geq 0}$. We refer to \cite{EN00} for a detailed account on these semigroups. Kadets and Shumyatskiy proved in \cite{KSadd} the following:
\begin{theorem}[Kadets, Shumyatskiy]
Let $(T(t))_{t\in\RR}$ be a $C_0$-group that is bounded, i.e.~$\|T(t)\|\leq M$ for all $t\in\RR$. Then for every 
for every $t_1,t_2\in \RR$
\begin{equation*}
\ker(T(t_1)-\Id)(T(t_2)-\Id)=\ker(T(t_1)-\Id)+\ker(T(t_2)-\Id).
\end{equation*}
\end{theorem}
 Surprisingly enough, the proof in \cite{KSadd}---a clever argument establishing the relative compactness of orbits of the vectors in $\ker(T(t_1)-\Id)(T(t_2)-\Id)$---does not directly work for more than $2$ operators, at least it is not immediately seen how. On the other hand, it essentially makes use of the invertibility of the operators $T(t_1)$ and $T(t_2)$. We consider two variations of the Kadets--Shumyatskiy result: 1) by allowing arbitrary finite number of operators but restricting to certain type of Banach spaces; 2) by dropping the assumption of the invertibility of the operators but restricting to certain type of semigroups.

To prove the validity of
\begin{equation*}
\ker(T(t_1)-\Id)\cdots (T(t_n)-\Id)=\ker(T(t_1)-\Id)+\cdots+\ker(T(t_n)-\Id)
\end{equation*} our strategy is to prove that a vector $x\in\ker(T(t_1)-\Id)\cdots(T(t_n)-\Id)$ belongs to a $(T(t))_{t\geq0}$-invariant subspace on which the semigroup operators are all mean ergodic. Thus finally one will be in the position to apply Corollary \ref{cor:mearg}.
\begin{definition}
Let $E$ be a Banach space and $(T(t))_{t\geq0}$ be a $C_0$-semigroup. A vector $x\in E$ is called \emph{asymptotically almost periodic} (with respect to this semigroup)  if the (forward) orbit
\begin{equation*}
O_T(x):=O(x):=\bigl\{T(t) x:t\geq 0\bigr\}
\end{equation*}
is relatively compact in $E$. Denote by $E_{\aap}$ the set of such vectors.
 If $(T(t))_{t\in\RR}$ is a $C_0$-group, then a vector is called \emph{almost periodic}, if both $O_T(x)$ and $O_{T^{-1}}(x)$ are relatively compact. These vectors constitute the set $E_{\ap}$.
\end{definition}
Suppose that  the $C_0$-semigroup $(T(t))_{t\geq 0}$ is bounded. It is then easy, and in fact well-known,  that $E_{\aap}\subseteq E$ is a closed subspace invariant under $T(t)$ for every $t\geq 0$. The analogous statement is true  for bounded $C_0$-groups and for $E_\ap$. 
 Trivially, for any $s> 0$ we have
\begin{equation*}
\ker(T(s)-\Id)\subseteq E_{\aap},
\end{equation*}
this being a consequence of the strong continuity of $(T(t))_{t\geq 0}$. By Example \ref{examp:cpt} the restricted semigroup to $E_{\aap}$ consists of mean ergodic operators. Let us record a consequence of this fact.
\begin{lemma}\label{l:apdecomp}
We have
\begin{equation*}
E_{\aap}\cap \ker(T(t_1)-\Id)\cdots(T(t_n)-\Id)=\ker(T(t_1)-\Id)+\cdots+\ker(T(t_n)-\Id).
\end{equation*}
\end{lemma}

 \begin{remark}\label{rem:omega0}
 \begin{enumerate}[1.]
 \item 
Let $\omega_0(T)\in [-\infty,\infty)$ be the \emph{growth bound} of the semigroup $T$, i.e., $\omega_0(T)=\frac1t\log r(T(t))$  for any $t>0$, where $r$ denotes the spectral radius. If $\omega_0(T)<0$, then $\|T(t)\|\to 0$ as $t\to \infty$, hence $E=E_{\aap}$, and anyway $\ker(T(t)-\Id)=\{0\}$ for each $t\geq 0$. So that we may assume $\omega_0(T)\geq 0$.
\item Also if $T(t)\to 0$ or $T(t)\to Q\in \LL(E)$ in the strong operator topology as $t\to \infty$, then again $E=E_\aap$. Hence Corollary \ref{cor:mearg} again applies, and this is also the case if the previous convergences hold in the weak operator topology. All in all we see that the decomposition problem is only interesting if the semigroup is not stable (or convergent)  in any usual sense.  
\end{enumerate}
 \end{remark}
The next---quite standard---lemma further enlightens the role of strong continuity. The proof is included here only for sake of completeness.
\begin{lemma}\label{lem:aapchar}
Let $(T(t))_{t\geq0}$ be a bounded  $C_0$-semigroup, and let $x\in E$. Then $x\in E_{\aap}$ holds if and only if \begin{equation*}O(t,x):=\{T(t)^nx=T(nt)x:n\in \NN\}\end{equation*} is relatively compact for some/all $t>0$.
\end{lemma}
\begin{proof}
Since $O(t,x)\subseteq \{T(t)x:t\geq0\}$, we have the relative compactness of $O(t,x)$ whenever $x\in E_{\aap}$. The converse implication follows from the fact that on bounded subsets of $\LLL(E)$ strong convergence coincides with uniform  convergence on relatively compact sets. More specifically, suppose $O(t,x)$ is relatively compact for some $t>0$. If $(t_n)\subset \RR_+$ is an arbitrary sequence then $t_n=q_nt+r_n$ with $q_n\in \NN_0$ and $r_n\in [0,t)$. By  strong continuity the set $\{T(s)x:s\in [0,t]\}$ is compact. So, also by the assumptions, there is a subsequence $(n_k)$ such that $T(q_{n_k}t)x$ and $T(r_{n_k})x$ are both convergent as $k\to\infty$.
We now have
\begin{align*}
&\|T(t_{n_k})x-T(t_{n_\ell})x\|=\|T(q_{n_k}t)T(r_{n_k})x-T(q_{n_\ell}t)T(r_{n_\ell})x\|\\
&\quad\leq \|T(r_{n_k})(T(q_{n_k}t)-T(q_{n_\ell}t))x\|+\|T(q_{n_\ell}t)(T(r_{n_k})x-T(r_{n_\ell})x)\|\\
&\quad\leq M\|(T(q_{n_k}t)x-T(q_{n_\ell}t)x)\|+M\|(T(r_{n_k})x-T(r_{n_\ell})x)\|\to 0\quad\mbox{as $k,\ell\to\infty$},
\end{align*}
hence $(T(t_{n_k})x)$ is a Cauchy sequence. Thus we have proved the relative compactness of the set $\{T(t)x:t\geq0\}$.
\end{proof}

\begin{lemma}\label{lem:apchar}
For a bounded $C_0$-group $(T(t))_{t\in \RR}$ and $x\in E_\aap$ 
\begin{equation*}
\bigl\{T(t)x:t\in \RR\bigr\}\subseteq E
\end{equation*}
is relatively compact, i.e., $E_\aap=E_\ap$.
\end{lemma}
\begin{proof}
Denote by $E_{\ap}$ the set of vectors for which the set under consideration is relatively compact. This is then a closed subspace of $E$ invariant under the semigroup operators. The assertion follows from the  Jacobs--de Leeuw--Glicksberg decomposition:
\begin{equation*}
E_{\aap}=E_{\ap}\oplus \{y:T(t)y\to 0\mbox{ for $t\to \infty$}\},
\end{equation*}
see, e.g., \cite[Ch.{} 15]{EFHN}. For bounded groups the second direct summand is $\{0\}$.
\end{proof}
Under certain  assumptions on the Banach space $E$ the Kadets, Shumyatskiy result generalizes for each $n\in \NN$:
\begin{theorem}\label{thm:nocodecomp}
Let $E$ be a Banach space that does not contain a copy of $\co$, and let $(T(t))_{t\in \RR}$ be a bounded $C_0$-group.
Then for every $n\in \NN$ and $t_1,\dots,t_n\in \RR$ we have
\begin{equation*}
\ker(T(t_1)-\Id)\cdots(T(t_n)-\Id)=\ker(T(t_1)-\Id)+\cdots+\ker(T(t_n)-\Id).
\end{equation*}
\end{theorem}

In order to prove this we shall need a special case of the Bohr--Kadets type result proved by Basit in \cite{Basit71}.

\begin{theorem}[Basit]\label{thm:basit}
Let $E$ be a Banach space, $F:\ZZ\to E$ be a bounded function with $F(1+\cdot)-F(\cdot)$ being almost periodic in $\ell^\infty(\ZZ;E)$. If $E$ contains no copy of $\co$, then $F$ is almost periodic.
\end{theorem}

This has the following implication for us.
\begin{proposition}\label{prop:lift}
Let $(T(t))_{t\in\RR}$ be a bounded $C_0$-group on the Banach space $E$, which contains no isomorphic copy of $\co$.
If for some $a>0$ the vector $(T(a)-\Id)x$ is asymptotically almost periodic, then so is $x$.
\end{proposition}
\begin{proof}
By Lemmas \ref{lem:aapchar} and \ref{lem:apchar} we have to prove that $O(a,x):=\{T(na)x:n\in \NN\}\subseteq E$
is a relatively compact set. The function $F$ defined by $F(n):=T(na)x$ for
$n\in \ZZ$ satisfies the assumptions of Theorem \ref{thm:basit}, so we obtain
that $F$ is almost periodic, hence its range $F(\ZZ)$ is relatively compact in $E$. 
\end{proof}

\begin{proof}[Proof of Theorem \ref{thm:nocodecomp}]
Suppose \begin{equation*}x\in \ker(T(t_1)-\Id)(T(t_2)-\Id)\cdots
(T(t_n)-\Id).\end{equation*}
We prove by induction that $x\in E_{\aap}$, then Lemma \ref{l:apdecomp} will finish the proof.
The case $n=1$ is trivial. Now let $y:=(T(t_1)-\Id)x$. Then $y$ satisfies a difference equation of length $n-1$ so by the inductive hypothesis $y$ is an asymptotically almost periodic vector. Proposition \ref{prop:lift} yields that $x$ itself belongs to $E_{\aap}$.
\end{proof}
We now return to the general Banach space setting, but restrict our attention to a special class of $C_0$-semigroups. Recall that  a $C_0$-semigroup $(T(t))_{t\geq 0}$ with exponential growth bound $\omega_0=\omega_0(T)$ is called \emph{norm continuous at infinity}, if
\begin{equation*}
\lim_{t\to\infty}\limsup_{h\downarrow 0}\bigl\|\ee^{-\omega_0t}T(t)\bigl(\Id-\ee^{-\omega_0 h}T(h)\bigr)\bigr\|=0.
\end{equation*}
We shall use two important facts:
\begin{theorem}[Spectral Mapping Theorem, \cite{mazon}]\label{thm:smt}
Let $(T(t))_{t\geq0}$ be an asymptotically norm continuous semigroup
with infinitesimal generator $A$ on  a Banach space $E$ and with $r(T(t))>0$. Then
 the relation
\begin{equation*}
\sigma(T(t))\cap \bigl\{\lambda\in \CC: |\lambda|=r(T(t))\bigr\}=\ee^{t\sigma(A)}\cap \bigl\{\lambda\in \CC: |\lambda|=r(T(t))\bigr\}\quad\mbox{($t\geq 0$)}
\end{equation*}
holds for the spectra.
\end{theorem}

\begin{theorem}[see {\cite[Thm.{} IV.2.26]{EN00}}]\label{thm:persgrp}
If $(T(t))_{t\geq0}$ is a strongly continuous $\alpha$-periodic semigroup (i.e., $T(\alpha)=\Id$ so
actually a group) with infinitesimal generator $A$, then $\sigma(A)\subset 2\pi
i\alpha\ZZ$ and the spectrum consist of eigenvalues only. i.e., $P\sigma(A)=\sigma(A)$.
\end{theorem}

Next lemma is an obvious modification of one in \cite{laczkovich/revesz:1989}, and was later proved also in \cite{KSave} in slightly more general situation. We recall the proof, because the result is interesting in itself.
\begin{lemma}\label{lem:reduce}
Let $T$ be a power-bounded operator on $E$. Then for $n\in \NN\setminus\{0\}$  the equality
\begin{equation*}
\ker(T-\Id)^n=\ker(T-\Id)
\end{equation*}
holds true.
\end{lemma}
\begin{proof}
Take $x\in \ker(T-\Id)^n$. We have to prove $x\in \ker(T-\Id)^{n-1}$. Let $n\geq 2$ and set $y=(T-\Id)^{n-2}x$, $z=(T-\Id)y$, so $Tz=z$.
 For any $N\in \NN$ we have
\begin{align*}
T^Ny&=\sum_{i=0}^{N-1} T^i(T-\Id)y+y=Ny+z.
\end{align*}
If $z\neq 0$, this yields a contradiction with the power-boundedness of $T$.%
\end{proof}

Recall that $t,s\in \RR$ are called \emph{incommensurable} if $s\RR\cap t\RR=\{0\}$; otherwise they are called \emph{commensurable}.

\begin{lemma}\label{normcontap}
Let $(T(t))_{t\geq 0}$ be a bounded strongly continuous semigroup that is
asymptotically norm continuous. Then for  every $t_1,t_2,\dots, t_n> 0$ 
\begin{equation*}
\ker(T(t_1)-\Id)(T(t_2)-\Id)\cdots
(T(t_n)-\Id)\subseteq E_{\aap}.
\end{equation*}
\end{lemma}
\begin{proof}
We may suppose $\omega_0(T)=0$ otherwise, the statement is trivial in view of Remark \ref{rem:omega0}.
We argue by induction on $n$. The case $n=1$ being trivial, let $n\in\NN$, $n\geq2 $. Take $x\in \ker(T(t_1)-\Id)(T(t_2)-\Id)\cdots(T(t_n)-\Id)$. We can assume that $t_1,t_2,\dots, t_n$ are
pairwise incommensurable. Indeed, if, say, $t_1$ and $t_2$ are
commensurable, then $m_1t_1=m_2t_2=s$ for some $m_1,m_2\in\ZZ\setminus\{0\}$, hence 
\begin{equation*}
T(s)-\Id=T(m_kt_k)-\Id=T(t_k)^{m_k}-\Id=(T(t_k)-\Id)\sum_{i=0}^{m_k-1}T(t_k)^{i}\quad \mbox{for $k=1,2$}.
\end{equation*}
So that
\begin{align*}
x\in& \ker\bigl(T(m_1t_1)-\Id)(T(m_2t_2)-\Id)(T(t_3)-\Id)\cdots(T(t_n)-\Id)\bigr)\\
&\quad=\ker\bigl((T(s)-\Id)^2(T(m_3t_3)-\Id)\cdots(T(t_n)-\Id)\bigr).
\end{align*}
Hence by Lemma \ref{lem:reduce} \begin{equation*}
x\in \ker(T(s)-\Id)(T(m_3t_3)-\Id)\cdots(T(t_n)-\Id),
\end{equation*}
and the induction hypothesis yields $x\in E_\aap$.

\medskip\noindent So we can assume that $t_1,t_2,\dots, t_n$ are all pairwise
incommensurable, and without loss of generality that $t_1=1$ and let $y:=(T(1)-\Id)x$. Then by the
inductive assumption $y$ is asymptotically almost periodic. By Lemma \ref{l:apdecomp} we can
decompose $y$ as
\begin{equation*}
y=y_2+y_3+\cdots +y_n,
\end{equation*}
with $y_j\in \ker(T(t_j)-\Id)$ for $j=2,3,\dots,n$. We restrict
the semigroup $(T(t))_{t\geq 0}$ to each of the spaces
$\ker(T(t_j)-\Id)$ on which it becomes a $t_j$-periodic $C_0$-group, denoted by $(S_j(t))_{t\in \RR}$.  Evidently $\omega_0(S_j)\leq \omega_0(T)=0$. If for some $j\in \{2,\dots, n\}$ strict inequality  $\omega_0(S_j)< \omega_0(T)=0$ holds, then $S_j(1)-\Id$ is invertible. Consider only $j\in \{2,\dots, n\}$ with $\omega_0(S_j)=0$. By Theorem \ref{thm:persgrp} we know that for the generator $A_j$ of $(S_j(t))_{t\geq 0}$
one has $\sigma(A_j)=P\sigma(A_j)\subset \frac{2\pi i}{t_j}\ZZ$. Also we
have the spectral mapping theorem in the following from
\begin{equation*}
\sigma(S_j(t))\cap \bigl\{z:|z|=1\bigr\}=\ee^{t\sigma(A_j)}\cap\bigl\{z:|z|=1\bigr\},
\end{equation*}
see Theorem \ref{thm:smt} with $r(S_j(t))=1$.  These and the incommensurability of $1$ and $t_j$ imply that
$1\not\in\sigma(S_j(1))$. Indeed, if $1\in\sigma(S_j(1))$ holds, then for some $k\in \ZZ$ we have $2\pi i k\in \sigma(A_j)\subseteq 2\pi i/t_j\ZZ$, contradicting the incommensurability.  So again $(S_j(1)-\Id)$ is invertible. Finally, we obtain for every $j\in \{1,\dots,n\}$ that  there is $y_j\in \ker(T(t_j)-\Id)$ with $(T(1)-\Id)x_j=y_j$.
Now $x_1:=x-(x_2+\cdots+x_n)$ belongs to $\ker(T(1)-\Id)$ by construction. This shows
that $x$ is an almost periodic vector.
\end{proof}

\begin{theorem}
Let $(T(t))_{t\geq 0}$ be a bounded$C_0$-semigroup that is
norm continuous at infinity. Then for all $n\in\NN$ and $t_1,\dots,t_n\geq0$ we have
\begin{equation*}
\ker(T(t_1)-\Id)\cdots (T(t_n)-\Id)=\ker(T(t_1)-\Id)+\cdots+\ker(T(t_n)-\Id).
\end{equation*}
\end{theorem}
\begin{proof}
The assertion follows from Lemmas \ref{l:apdecomp} and \ref{normcontap}.
\end{proof}

\parindent0pt
\def\cprime{$'$} \def\cprime{$'$}
\providecommand{\bysame}{\leavevmode\hbox to3em{\hrulefill}\thinspace}
\providecommand{\MR}{\relax\ifhmode\unskip\space\fi MR }
\providecommand{\MRhref}[2]{%
  \href{http://www.ams.org/mathscinet-getitem?mr=#1}{#2}
}
\providecommand{\href}[2]{#2}

\end{document}